\documentclass[11pt,reqno]{amsart}

\usepackage[margin=1.1in]{geometry}
\usepackage{amsmath,amssymb,amsthm}
\usepackage{enumitem}
\usepackage[colorlinks=true,linkcolor=blue,citecolor=blue,urlcolor=blue]{hyperref}

\theoremstyle{plain}
\newtheorem{theorem}{Theorem}[section]
\newtheorem{proposition}[theorem]{Proposition}
\newtheorem{lemma}[theorem]{Lemma}
\newtheorem{corollary}[theorem]{Corollary}
\newtheorem{introthm}{Theorem}

\theoremstyle{definition}
\newtheorem{definition}[theorem]{Definition}
\newtheorem{example}[theorem]{Example}
\newtheorem{problem}[theorem]{Problem}

\theoremstyle{remark}
\newtheorem{remark}[theorem]{Remark}

\newcommand{\Fp}{\mathbb{F}_p}
\newcommand{\Z}{\mathbb{Z}}
\newcommand{\Aut}{\operatorname{Aut}}
\newcommand{\Out}{\operatorname{Out}}
\newcommand{\UAut}{\operatorname{UAut}}
\newcommand{\Hom}{\operatorname{Hom}}
\newcommand{\End}{\operatorname{End}}
\newcommand{\Der}{\operatorname{Der}}
\newcommand{\gr}{\operatorname{gr}}
\newcommand{\im}{\operatorname{im}}
\newcommand{\ad}{\operatorname{ad}}
\newcommand{\I}{\mathrm{I}}
\newcommand{\sh}{\psi}
\newcommand{\Heis}{\operatorname{Heis}}

\newcommand{\fl}{\mathfrak l}
\newcommand{\ff}{\mathfrak f}
\newcommand{\fr}{\mathfrak r}
\newcommand{\fe}{\mathfrak e}
\newcommand{\fv}{\mathfrak v}
\newcommand{\fu}{\mathfrak u}
\newcommand{\Om}{\Omega}

\title[Unipotent automorphisms and the Lazard Lie algebra]{Unipotent automorphisms of modular group algebras and the homology of the Lazard Lie algebra}
\author{Diego Tuzzolo}
\address{S\~ao Paulo, SP, Brazil}
\email{diegotuzzolo@usp.br}
\date{September 9, 2026}


\begin{document}

\begin{abstract}
Let $p$ be an odd prime and $H$ a finite $p$-group. Hertweck and Soriano
parametrize the central Frattini extensions of $H$ by subspaces of a relation
module $V$, matching isomorphism classes of extensions with orbits of
$\Out(H)$ and isomorphism classes of the associated small group rings with
orbits of $\Out(\Fp H)$. The discrepancy between the two is measured by the
group $U(H)\le GL(V)$ of displacements induced by the normalized unipotent
automorphisms of $\Fp H$; it is the mechanism behind the counterexamples to
the modular isomorphism problem in characteristic $2$. We prove that for $H$
of exponent $p$ and class at most $p-2$ every such automorphism acts on $V$
exactly as an automorphism of $H$ does, and that
$U(H)=\exp\Om\bigl(\Der_{>0}(\fl(H))\bigr)$, where $\fl(H)$ is the Lazard Lie
algebra of $H$, $\Der_{>0}$ its derivations raising the lower central
filtration, and $\Om$ the natural action of derivations on $H_2(\fl(H))$.
Hence $U(H)=1$ precisely when every derivation of positive degree acts
trivially on $H_2(\fl(H))$. We compute $U(H)$ for all groups of class $2$ and
exponent $p$, obtaining $p^{2n(n-1)}$ for the Heisenberg groups
$\Heis(p^n)$.
\end{abstract}

\maketitle

\section{Introduction}\label{sec:intro}

Does the multiplication table of a finite group survive when it is
linearized? Concretely: if $G$ and $\widetilde G$ are finite $p$-groups whose
group algebras over the field of $p$ elements are isomorphic as algebras,
must $G$ and $\widetilde G$ be isomorphic? This is the \emph{modular
isomorphism problem}, open since the 1950s. The algebra $\Fp G$ remembers a
great deal --- its dimension is $|G|$, and its radical filtration recovers a
canonical central series of $G$ --- but it forgets the group basis, and the
question is whether what it remembers is enough.

For half a century the evidence pointed one way: the answer is positive for
abelian groups \cite{Des56}, for groups of order at most $p^5$
\cite{Pas65,SS96}, for metacyclic groups \cite{Bag88,San96}, for groups of
class $2$ with elementary abelian derived subgroup \cite{San89}, and in
several further classes surveyed in \cite{Mar22}. Then Garc\'ia-Lucas,
Margolis and del R\'io \cite{GLMdR22} produced two non-isomorphic groups of
order $2^9$ with isomorphic modular group algebras. Their counterexample is
specific to characteristic $2$, and for odd primes the problem remains open,
with the known invariants behaving markedly better \cite{GdRS23,GLdR25}.
Understanding \emph{why} the two characteristics differ is therefore not a
technicality but the central question in the area, and it is the question
this paper answers for one precisely delimited mechanism.

That mechanism is the construction through which the counterexamples were
found. Hertweck and Soriano \cite{HS07} fix a $p$-group $H$ and study all
central extensions $1\to A\to G\to H\to1$ with $A$ inside the Frattini
subgroup of $G$: these are the \emph{central Frattini extensions} of the
\emph{base} $H$, and they are all quotients of a single universal one,
$1\to V\to E\to H\to1$, so that each is recorded by a subspace of the
\emph{relation module} $V$. Two extensions are isomorphic as groups exactly
when the corresponding subspaces lie in one orbit of $\Aut(H)$. Now replace
the group by an algebra. To each extension Hertweck and Soriano attach a
small quotient of $\Fp G$, the \emph{small group ring}
$s(G,A)=\Fp G/\I(\Fp A)\I(\Fp G)$, and prove that two extensions have
isomorphic small group rings exactly when the corresponding subspaces lie in
one orbit of the automorphism group of the \emph{algebra} $\Fp H$. Since an
algebra has more automorphisms than a group, the second orbit relation is
coarser, and any genuine coarsening is a pair of non-isomorphic groups with
isomorphic small group rings --- the first step towards a counterexample. This
is exactly what happens for $H=D_8$ \cite[Ex.~6.3]{HS07}.

The coarsening is measured by a single group. An automorphism of $\Fp H$ is
\emph{normalized unipotent} if it induces the identity on
$\I/\I^2\cong H/\Phi(H)$; these form a group $\UAut(\Fp H)$, and we write
$U(H)$ for its image in $GL(V)$. Group automorphisms contribute to $U(H)$ as
well --- those that are trivial modulo the Frattini subgroup --- and what one
wants to know is whether they contribute everything:
\begin{quote}
\emph{if every element of $U(H)$ is realized by an automorphism of the group
$H$, then two central Frattini extensions of $H$ with isomorphic small group
rings are isomorphic groups.}
\end{quote}
Hertweck and Soriano computed $U(H)$ for elementary abelian bases and for
quotients of free groups by terms of the Jennings series; beyond that the
group has been computed case by case, by hand, through the flag structure of
$V$. Our purpose is to show that in odd characteristic, and below the degree
where the prime interferes, $U(H)$ is not a group-ring invariant at all: it is
an invariant of a nilpotent Lie algebra, and it is computed by graded linear
algebra.

\subsection*{The main theorem}
Let $p$ be odd and let $H$ have exponent $p$ and nilpotency class $c<p$. The
Lazard correspondence attaches to $H$ a nilpotent Lie algebra $\fl=\fl(H)$
over $\Fp$ of the same dimension and class, with $\Aut(H)=\Aut(\fl)$.
Present it as $\fl=\ff/\fr$ with $\ff$ free nilpotent on $d=d(H)$
generators; by Hopf's formula the second homology is
$H_2(\fl)=\fr/[\ff,\fr]$. A derivation $\mu$ of $\fl$ lifts to $\ff$,
preserves $\fr$ and $[\ff,\fr]$, and so induces an endomorphism $\Om(\mu)$
of $H_2(\fl)$ which does not depend on the lift; write $\Der_{>0}(\fl)$ for
the derivations $\mu$ with $\mu(\fl)\subseteq\gamma_2(\fl)$, which are
nilpotent, so that $\exp(\mu)$ is an automorphism of $\fl$.

\begin{introthm}\label{thmA}
Let $p$ be odd and $H$ a finite $p$-group of exponent $p$ and class
$c\le p-2$. Then:
\begin{enumerate}[label=\textup{(\arabic*)}]
\item every $\alpha\in\UAut(\Fp H)$ acts on $V$ exactly as the group
automorphism $\exp(\mu)$ does, for some $\mu\in\Der_{>0}(\fl(H))$;
\item the commutator sector $C(V)=V\cap[E,E]$ is canonically isomorphic to
$H_2(\fl(H))$, the group $U(H)$ acts faithfully on it, and
\[
U(H)=\exp\Bigl(\Om\bigl(\Der_{>0}(\fl(H))\bigr)\Bigr)\ \le\
GL\bigl(H_2(\fl(H))\bigr),\qquad
|U(H)|=p^{\,\dim_{\Fp}\Om(\Der_{>0}(\fl(H)))};
\]
\item $U(H)=1$ if and only if every derivation of positive degree of
$\fl(H)$ acts trivially on $H_2(\fl(H))$;
\item if $\fl(H)$ is graded, then non-isomorphic central Frattini extensions
of $H$ have non-isomorphic small group rings.
\end{enumerate}
\end{introthm}

Two ingredients carry the proof, and both are elementary once found. The
first is a logarithm: a normalized unipotent automorphism of the truncated
algebra $\Fp H/\I^{\,p}$ has $\log\alpha$ well defined, and this is a genuine
derivation of the algebra in spite of the characteristic
(Lemma~\ref{lem:log}). The second is a projection. Inside the small group
ring of the universal extension, truncated in degree $p$, the Lazard algebra
$\fe$ of that extension sits as a direct summand of a decomposition by
symmetrized powers, and every summand is stable under $\ad(\fe)$
(Lemma~\ref{lem:symm}). Composing the logarithm with the projection onto
$\fe$ produces a derivation of $\fe$; its exponential is the group
automorphism of Theorem~\ref{thmA}(1). The formula in (2) then follows,
since both inclusions are cheap once (1) is available.

It is worth recording what makes this work, because it is also what fails at
$p=2$. The defining truncation $\{A-1\}\cdot\I=0$ of the small group ring
forces $\log v=v-1$ for $v\in V$, and thereby places the relation module
\emph{literally inside} the Lazard algebra; the centrality $[V,E]=1$ of the
universal extension makes the relevant ideal compatible with symmetrization.
Everything killed by the projection --- the quadratic and higher parts of a
unipotent automorphism, which is where the extra freedom of the algebra over
the group lives --- turns out to be invisible on $V$. In characteristic $2$
the three ingredients fail simultaneously: $\binom p2$ is odd, $1/2$ does not
exist, and degree $p$ swallows the whole flag of $V$. The known
counterexamples live exactly there.

\subsection*{Consequences}
Theorem~\ref{thmA} turns the computation of $U(H)$ into a question about
derivations acting on the homology of a nilpotent Lie algebra, and the degree
of $H_2$ alone often settles it.

\begin{introthm}\label{thmB}
Let $p\ge5$.
\begin{enumerate}[label=\textup{(\alph*)}]
\item \textup{(one layer)} If $H_2(\fl(H))$ is concentrated in a single
degree, then $U(H)=1$. This covers the relatively free groups
$F_d/\gamma_{c+1}(F_d)F_d^{\,p}$ for all $d\ge2$ and $2\le c\le p-2$, in
particular the extraspecial group $p^{1+2}_+$ of order $p^3$ and exponent
$p$.
\item \textup{(two generators)} If $H=F_2/\gamma_{c+1}(F_2)F_2^{\,p}N$ with
$3\le c\le p-2$ and $N\le\gamma_c(F_2)$, then $U(H)=1$; in particular this
holds for the group of maximal class, order $p^4$ and exponent $p$.
\item \textup{(class two)} If $H$ has class $2$ and exponent $p$, with
$L_1=H/\Phi(H)$, $L_2=H'$ and commutator map $\omega:\Lambda^2L_1\to L_2$,
then $U(H)\cong\Hom(L_1,L_2)/\ker\Om$ is elementary abelian, every element of
it is realized by a central automorphism of $H$, and $\ker\Om$ contains the
inner derivations. For the Heisenberg group $\Heis(p^n)$, the Sylow
$p$-subgroup of $SL_3(p^n)$, and for the direct power $(p^{1+2}_+)^M$ one has
$|U(H)|=p^{2n(n-1)}$, respectively $p^{2M(M-1)}$.
\end{enumerate}
In all these cases non-isomorphic central Frattini extensions of $H$ have
non-isomorphic small group rings.
\end{introthm}

Part (b) covers, in one line of Jacobi's identity, a base whose treatment by
the flag structure of $V$ requires a computation in degree $5$ of a
Poincar\'e--Birkhoff--Witt basis. Part (c) requires no hypothesis on
$\omega$ whatever, and so applies to all the semifield groups of Verardi
\cite{Ver87}, among them the families of Lewis and Wilson \cite{LW12} of
$p^{n^2/24}$ groups of order $p^n$ sharing a character table --- precisely
where ordinary character theory is blind and only the modular group algebra
can separate.

\subsection*{Conventions}
Throughout, $p$ is a prime and $K=\Fp$; all groups are finite $p$-groups.
$\I=\I(KG)$ is the augmentation ideal, $D_n(G)=G\cap(1+\I^n)$ the Jennings
series, $L_n(G)=D_n(G)/D_{n+1}(G)$ and $L(G)=\bigoplus_nL_n(G)$. Commutators
are $[g,h]=g^{-1}h^{-1}gh$; in associative algebras $[a,b]=ab-ba$. We write
$\gamma_n$ for the lower central series and $\mho_k(G)=\langle g^{p^k}
\rangle$. Free groups are $F=F_d$; free Lie algebras over $\Fp$ are
$\ell=\ell(d)=\bigoplus_k\ell_k$, and $\ell_k(d)=\dim\ell_k$.

\section{Preliminaries}\label{sec:prelim}

\subsection{Jennings, Quillen, and graded rigidity}
By Jennings' theorem \cite{Jen41},
\begin{equation}\label{eq:jennings}
D_n(G)=\prod_{ip^j\ge n}\mho_j(\gamma_i(G)),
\end{equation}
and $L(G)$ is a graded restricted Lie algebra under commutation and $p$-th
powers, generated in degree one. By Quillen \cite{Qui68},
$\gr(KG)=\bigoplus_n\I^n/\I^{n+1}$ is the restricted enveloping algebra
$u(L(G))$, the map $L_n(G)\to\I^n/\I^{n+1}$, $gD_{n+1}\mapsto g-1+\I^{n+1}$,
being an embedding. For $G$ abelian, \eqref{eq:jennings} reads
$D_n(G)=\mho_{\lceil\log_pn\rceil}(G)$.

\begin{lemma}[graded rigidity]\label{lem:grrigid}
Let $\Theta:KG\to K\widetilde G$ be a $K$-algebra isomorphism. Then
$\gr\Theta:u(L(G))\to u(L(\widetilde G))$ maps $L_n(G)$ onto
$L_n(\widetilde G)$ for every $n$, compatibly with brackets and $p$-th
powers.
\end{lemma}
\begin{proof}
$\Theta$ preserves the radical, hence the powers $\I^n$, so $\gr\Theta$ is an
isomorphism of graded algebras; it maps $L_1(G)=\I/\I^2$ onto
$L_1(\widetilde G)$, hence the restricted Lie subalgebra generated by
$L_1(G)$, which is $L(G)$, onto $L(\widetilde G)$, preserving degrees,
brackets and $p$-th powers (these being the associative ones).
\end{proof}

\subsection{The Hertweck--Soriano frame}\label{subsec:HS}
Fix $H$ with $d=d(H)$ and a minimal presentation $1\to R\to F\to H\to1$, so
that $R\le\Phi(F)=F^p[F,F]$. Put
\[
L=[F,R]R^p,\qquad E=F/L,\qquad V=R/L.
\]
Then $V$ is elementary abelian, central in $E$ and contained in $\Phi(E)$,
and $1\to V\to E\to H\to1$ is universal among central Frattini extensions
$1\to A\to G\to H\to1$ (those with $A\le\Phi(G)$): each such is $E/U\to H$
for a subspace $U\le V$. We write $\mathcal F(H)$ for this class. The
\emph{small group ring} of $(G,A)$ is $s(G,A)=KG/\I(KA)\I(KG)$, together with
the ideal $\{A-1\}$, the image of $\I(KA)KG$; one has $\{A-1\}\cong A$
canonically, $\{A-1\}$ is central, and
\begin{equation}\label{eq:trunc0}
\{A-1\}\cdot\I(s(G,A))=\I(s(G,A))\cdot\{A-1\}=0
\end{equation}
\cite[(3.1)]{HS07}. Put $s=s(E,V)$ and $W_n=V\cap D_n(E)$, the \emph{depth
flag} of $V$; note $W_2=V$. Over the prime field, Passi and Sehgal's
recognition of the Jennings ideals \cite{PS72} gives, for every
$(G,A)\in\mathcal F(H)$,
\begin{equation}\label{eq:PS}
\{A-1\}\cap\I(s(G,A))^n=\{A\cap D_n(G)-1\}\qquad(n\ge1),
\end{equation}
see \cite[\S3]{HS07}. Every $\alpha\in\Aut(KH)$ lifts to an automorphism
$\hat\alpha$ of $s(E,V)$ preserving $\{V-1\}$ \cite[(4.9), Rem.~4.5]{HS07},
hence restricts to an automorphism of $V=1+\{V-1\}$, independent of the lift
(Lemma~\ref{lem:welldef}); inner automorphisms act trivially, and we write
$\rho:\Out(KH)\to GL(V)$ and $v\mapsto v^\alpha$.

\begin{theorem}[{\cite[Thm.~4.4, Cor.~3.7]{HS07}}]\label{thm:HS}
Isomorphism classes of extensions in $\mathcal F(H)$ correspond to
$\rho(\Out(H))$-orbits of subspaces of $V$, and isomorphism classes of the
pairs $(s(G,A),\{A-1\})$ correspond to $\rho(\Out(KH))$-orbits.
\end{theorem}

Since $KH$ is local, every algebra automorphism preserves $\I$ and acts on
$\I/\I^2\cong H/\Phi(H)$. Following \cite[Def.~5.1]{HS07}, $\UAut(KH)$ is the
group of automorphisms inducing the identity on $\I/\I^2$; the lift
$\hat\alpha$ of $\alpha\in\UAut(KH)$ satisfies
$(\hat\alpha-1)\I(s)\subseteq\I(s)^2$ \cite[Lem.~5.2]{HS07}. Every
$\beta\in\Aut(H)$ extends to $KH$, and $\beta\in\UAut(KH)$ if and only if
$\beta$ is the identity on $H/\Phi(H)$.

\begin{lemma}\label{lem:welldef}
The action of $\alpha\in\Aut(KH)$ on $V$ does not depend on the lift.
\end{lemma}
\begin{proof}
Two lifts differ by an element of $\Hom(H/\Phi(H),V)$ \cite[Rem.~4.5]{HS07},
acting on the images $\hat x_i\in E\subset s^\times$ of a free basis by
$\hat x_i\mapsto\hat x_iv_i$ with $v_i\in V$ central. If $v$ is the class of
a relator $r$ then $r(\hat x_1v_1,\dots)=r(\hat x)\prod_iv_i^{e_i(r)}$, where
$e_i(r)$ is the exponent sum of $r$ in $x_i$; now $r\in\Phi(F)$ gives
$p\mid e_i(r)$, and $V$ is elementary abelian.
\end{proof}

\begin{definition}
$U(H)$ is the image of $\UAut(KH)$ in $GL(V)$. The \emph{Levi factorization}
is the statement $\Aut(KH)=\Aut(H)\cdot\UAut(KH)$; it holds if and only if
the images of $\Aut(H)$ and of $\Aut(KH)$ in $GL(H/\Phi(H))=GL(\I/\I^2)$
coincide.
\end{definition}

\begin{lemma}\label{lem:reform}
\begin{enumerate}[label=\textup{(\arabic*)}]
\item For $\alpha\in\UAut(KH)$ and $v\in W_n$ one has
$v^\alpha v^{-1}\in W_{n+1}$; in particular $U(H)$ is a $p$-group of
unipotent transformations stabilizing the depth flag.
\item $\rho(\Out(KH))\supseteq\rho(\Out(H))\cdot U(H)$, with equality under
the Levi factorization.
\item If the Levi factorization holds and every element of $U(H)$ lies in
$\rho(\im(\Aut(H)\to\Out(KH)))$, then $\rho(\Out(KH))=\rho(\Out(H))$, and
for $G,\widetilde G\in\mathcal F(H)$ an isomorphism of pairs
$(s(G,A),\{A-1\})\cong(s(\widetilde G,\widetilde A),\{\widetilde A-1\})$
implies $G\cong\widetilde G$.
\end{enumerate}
\end{lemma}
\begin{proof}
(1) $\hat\alpha$ preserves $\I(s)^n$ and $\{V-1\}$, hence each $\{W_n-1\}$ by
\eqref{eq:PS}; since $(\hat\alpha-1)\I^n\subseteq\I^{n+1}$, on $\{W_n-1\}$
this reads $\{v^\alpha-1\}-\{v-1\}=\{v^\alpha v^{-1}-1\}\in\I^{n+1}$, that
is, $v^\alpha v^{-1}\in W_{n+1}$. (2) Group automorphisms extend to $KH$,
and under the Levi factorization every class of $\Out(KH)$ is a class of
$\Out(H)$ times a class of $\UAut(KH)$. (3) is (2) together with
Theorem~\ref{thm:HS}.
\end{proof}

\subsection{The Lazard correspondence}\label{subsec:lazard}
Let $p$ be odd. The Lazard correspondence \cite{Laz54},
\cite[Ch.~10]{Khu98} is an equivalence between the category of finite
$p$-groups of class $<p$ and that of nilpotent Lie rings of class $<p$ and
$p$-power order, under which a group of exponent $p$ corresponds to a
nilpotent Lie algebra over $\Fp$ of the same dimension $\log_p|G|$ and the
same class. We write $\fl(G)$ for the Lie algebra of $G$ and use that the
correspondence is functorial ($\Aut(G)=\Aut(\fl(G))$, homomorphisms
corresponding to homomorphisms), that normal subgroups correspond to ideals
and quotients to quotients, that $\gamma_n(G)$ corresponds to
$\gamma_n(\fl(G))$ and $[N,M]$ to $[\mathfrak n,\mathfrak m]$, and that the
relatively free group of exponent $p$ and class $k<p$ on $d$ generators
corresponds to the free nilpotent Lie algebra of class $k$ on $d$ generators
(the group defined by the Baker--Campbell--Hausdorff law on the latter has
the required universal property).

\subsection{Presentations and the action on homology}\label{subsec:Omega}
Let $\fl$ be a nilpotent Lie algebra over $\Fp$ of class $c$ with
$d=\dim\fl/\gamma_2(\fl)$, and let $\ff$ be the free nilpotent Lie algebra of
class $c+1$ on $d$ generators; choose a surjection $\ff\to\fl$ inducing an
isomorphism on abelianizations, with kernel $\fr$, so that
$\fr\subseteq\gamma_2(\ff)$. Put
\[
\fe:=\ff/[\ff,\fr],\qquad\fv:=\fr/[\ff,\fr]\subseteq\fe,
\]
so that $\fv$ is a central ideal of $\fe$ and $\fe/\fv=\fl$. Since
$\gamma_{c+1}(\ff)\subseteq\fr$, the ideal $[\ff,\fr]$ contains
$\gamma_{c+2}$ of the free Lie algebra, and Hopf's formula for Lie algebras
\cite[Ch.~VII]{HiSt97} gives $\fv\cong H_2(\fl)$. The filtration
$\fv\cap\gamma_n(\fe)$, $n\ge2$, of $\fv$ is its \emph{degree filtration}.

Every $\theta\in\Aut(\fl)$ lifts to an endomorphism of $\ff$ (choose images of
the generators), which is an automorphism because it is one modulo
$\gamma_2(\ff)$; it preserves $\fr$ and $[\ff,\fr]$, hence induces an
automorphism of $\fe$ preserving $\fv$. Likewise every $\mu\in\Der(\fl)$
lifts to a derivation $\tilde\mu$ of $\ff$ (choose values on the
generators): writing $\pi:\ff\to\fl$, the maps $\pi\tilde\mu$ and $\mu\pi$
are $\pi$-derivations $\ff\to\fl$ agreeing on generators, hence equal, so
$\tilde\mu(\fr)\subseteq\fr$ and then
$\tilde\mu[\ff,\fr]\subseteq[\ff,\fr]$.

\begin{lemma}\label{lem:Omega}
\begin{enumerate}[label=\textup{(\arabic*)}]
\item Any two automorphisms of $\fe$ preserving $\fv$ and inducing the same
automorphism of $\fl$ agree on $\fv$; the same holds for derivations of $\fe$
preserving $\fv$ and inducing the same derivation of $\fl$.
\item Consequently $\Aut(\fl)$ acts on $\fv$, and there is a homomorphism of
Lie algebras $\Om:\Der(\fl)\to\End(\fv)$, $\Om(\mu)=\tilde\mu|_\fv$.
\item Inner derivations lie in $\ker\Om$.
\item If $\Theta=\exp(\tilde\mu)\in\Aut(\fe)$ for a lift $\tilde\mu$ of a
nilpotent $\mu\in\Der(\fl)$ with $\exp(\mu)$ defined, then
$\Theta|_\fv=\exp(\Om(\mu))$.
\end{enumerate}
\end{lemma}
\begin{proof}
(1) If $\Theta_1,\Theta_2$ are such automorphisms then
$\Theta_1^{-1}\Theta_2$ is $x\mapsto x+\phi(x)$ with $\phi:\fe\to\fv$, and
since $\fv$ is central, $\phi([x,y])=0$; as $\fv\subseteq\gamma_2(\fe)$
(because $\fr\subseteq\gamma_2(\ff)$), $\phi$ vanishes on $\fv$. For
derivations, the difference is a derivation $\delta$ of $\fe$ with values in
$\fv$, whence $\delta[x,y]=[\delta x,y]+[x,\delta y]=0$ and $\delta$
vanishes on $\gamma_2(\fe)\supseteq\fv$. (2) follows, using that
$[\tilde\mu,\tilde\mu']$ lifts $[\mu,\mu']$. (3) $\ad_x$ lifts to
$\ad_{\tilde x}$, which maps $\fr$ into $[\ff,\fr]$. (4) is immediate from
(1).
\end{proof}

\section{Exponent $p$: the universal extension and the dictionary}
\label{sec:expp}

From now on $p$ is odd and $H$ has exponent $p$ and class $c\le p-2$, so
that $D_n(H)=\gamma_n(H)$ and $L(H)=\gr\fl(H)$, where the Lazard algebra
$\fl=\fl(H)$ carries the lower central filtration. We keep the data
$\fl=\ff/\fr$, $\fe$, $\fv$ of \S\ref{subsec:Omega}, with $\ff$ free
nilpotent of class $c+1$ on the $d$ free generators, and write
\[
C(V):=V\cap[E,E]=\ker(\varphi|_V),\qquad \varphi:E\to\bar E'=E/[E,E],
\]
for the \emph{commutator sector} of $V$.

\begin{proposition}[structure of $E$ and $V$]\label{prop:E}
\begin{enumerate}[label=\textup{(\arabic*)}]
\item $E$ has class $\le c+1\le p-1$, $\gamma_2(E)$ has exponent $p$, and
$g\mapsto g^p$ is a homomorphism $E\to V$ factoring through $H/\Phi(H)$; its
image $\mho_1(E)$ is spanned by the independent classes of
$x_1^p,\dots,x_d^p$, and
\[
V=C(V)\oplus\mho_1(E),\qquad \dim V=\dim C(V)+d .
\]
\item $D_n(E)=\gamma_n(E)\mho_1(E)$ for $n\le p$ and $D_{p+1}(E)=1$; hence
$W_p=\mho_1(E)$, $W_{p+1}=1$, and $W_n\cap C(V)=\gamma_n(E)\cap C(V)$ for
$n\le p$, with $W_p\cap C(V)=0$.
\item Put $E^\circ:=E/\mho_1(E)$ and $V^\circ:=V/\mho_1(E)\cong C(V)$. Then
$E^\circ$ has exponent $p$ and class $\le c+1$, and under the Lazard
correspondence $E^\circ\leftrightarrow\fe$, $V^\circ\leftrightarrow\fv$ and
$E^\circ/V^\circ=H\leftrightarrow\fl$; the depth filtration of $C(V)$
corresponds to the degree filtration of $\fv$.
\end{enumerate}
\end{proposition}
\begin{proof}
(1) As $E/V=H$ has class $c$, we get $\gamma_{c+1}(E)\le V$ central and $E$
of class $\le c+1<p$; also $g^p\in V$ for all $g$, so $\mho_1(E)\le V$ is
central and elementary abelian. Hall--Petrescu \cite[Thm.~9.4]{Khu98} gives,
for $a,b\in E$,
\begin{equation}\label{eq:HP}
[a^p,b]=[a,b]^p\prod_{k=2}^{p}w_k^{\binom pk},\qquad
w_k\in\gamma_k(\langle a,[a,b]\rangle),
\end{equation}
where the left-hand side is $1$ ($a^p$ being central) and
$w_p\in\gamma_p(E)=1$. We show $\mho_1(\gamma_k(E))=1$ by downward induction
on $k\ge2$: for $a\in\gamma_{k-1}$ and $b\in E$, all the $w_k$ in
\eqref{eq:HP} lie in $\gamma_{k+1}(E)$, whose $p$-th powers are trivial by
induction, so $[a,b]^p=1$; and for $g,h\in\gamma_k$ one has
$(gh)^p=g^ph^p\prod w_k^{\binom pk}$ with $w_k\in\gamma_{2k}\subseteq
\gamma_{k+1}$, so $\gamma_k$ has exponent $p$. The same expansion for
$g,h\in E$, with $w_k\in\gamma_2(E)$ of exponent $p$ and $p\mid\binom pk$ for
$k<p$, shows that $g\mapsto g^p$ is a homomorphism; it kills $\gamma_2(E)$
and $\mho_1(E)$, hence factors through $H/\Phi(H)$, and its image is spanned
by the $x_i^p$. Since $R\gamma_2(F)=F^p\gamma_2(F)$ and the exponent vectors
of $L$ form $p^2\Z^d$, we get $\bar E'\cong(C_{p^2})^d$, in which the
$\bar x_i^p$ are independent; so the $x_i^p$ are independent, $\varphi$ is
injective on $\mho_1(E)$, and $V=C(V)\oplus\mho_1(E)$.

(2) By \eqref{eq:jennings}, for $n\le p$ the relevant factors are
$\gamma_n(E)$ and $\mho_1(\gamma_1)=\mho_1(E)$, the others being trivial by
(1); for $n=p+1$ all factors are trivial. Since $\mho_1(E)\le V$ we get
$W_n=(V\cap\gamma_n(E))\mho_1(E)$ for $n\le p$; if $v\in C(V)\cap W_n$, write
$v=gu$ with $g\in V\cap\gamma_n$ and $u\in\mho_1(E)$, so that $g\in C(V)$ and
$u\in C(V)\cap\mho_1(E)=0$. Finally $\gamma_p(E)=1$.

(3) $E^\circ=F/[F,R]R^pF^p=F^\circ/[F^\circ,\bar R]$, where
$F^\circ=F/\gamma_{c+2}(F)F^p$ is relatively free of exponent $p$ and class
$c+1$ (note $\gamma_{c+2}(F)\le[F,\gamma_{c+1}(F)]\le[F,R]$ and
$R^p\le F^p$) and $\bar R$ is the image of $R$. By \S\ref{subsec:lazard},
$F^\circ\leftrightarrow\ff$, $\bar R\leftrightarrow\fr$ and
$[F^\circ,\bar R]\leftrightarrow[\ff,\fr]$, whence
$E^\circ\leftrightarrow\fe$ and
$V^\circ=\bar R/[F^\circ,\bar R]\leftrightarrow\fv$, with $\gamma_n$
corresponding to $\gamma_n$; by (2) the depth filtration of
$C(V)\cong V^\circ$ is $C(V)\cap\gamma_n(E)$.
\end{proof}

The next proposition is the only place where we use the abelian quotient of
the universal extension. It is a Frobenius argument: the shadow of a
displacement in $\bar E'$ is a $p$-th power of something already deep, hence
trivial.

\begin{proposition}[displacements are commutators]\label{prop:absorb}
For every $\alpha\in\UAut(KH)$ and every $v\in V$ one has
$v^\alpha v^{-1}\in C(V)$. Consequently $U(H)$ acts trivially on $V/C(V)$
and on $\mho_1(E)$, and $U(H)\to GL(C(V))$ is injective.
\end{proposition}
\begin{proof}
Let $\bar V'=\varphi(V)$ and $A:=s(\bar E',\bar V')=K\bar E'/\I(K\bar V')
\I(K\bar E')$. Every element of $\bar V'$ is a $p$-th power in $\bar E'$,
because relators have exponent sums divisible by $p$; so
$\bar V'\le\mho_1(\bar E')=\Phi(\bar E')$ and
$(\bar E',\bar V')\in\mathcal F(H^{\mathrm{ab}})$, whence \eqref{eq:PS} holds
in $A$. Moreover $\varphi$ carries the generators $(v-1)(g-1)$ of
$\I(KV)\I(KE)$ to generators of the corresponding ideal over $\bar E'$, so it
induces an algebra morphism $\sh:s(E,V)\to A$ with
$\sh|_V=\varphi|_V$.

Let $\hat\alpha$ be a lift of $\alpha$, put
$\rho_i=\hat\alpha(\hat x_i)-\hat x_i$, which lies in $\I(s)^2$
\cite[Lem.~5.2]{HS07}, and $\sigma_i=\sh(\hat x_i)^{-1}\sh(\rho_i)\in
\I(A)^2$. Let $v\in V$ be the class of a relator $r$, so that
$v=r(\hat x_1,\dots,\hat x_d)$ and
$v^\alpha=r(u_1,\dots,u_d)$ with $u_i=\hat x_i(1+\hat x_i^{-1}\rho_i)$.
Applying $\sh$ and evaluating the word $r$ in the commutative algebra $A$ by
exponent sums,
\[
\sh(v^\alpha)=\prod_i\sh(\hat x_i)^{e_i(r)}\prod_i(1+\sigma_i)^{e_i(r)}
=\sh(v)\prod_i(1+\sigma_i)^{e_i(r)} .
\]
Since $\{V-1\}$ is a square-zero central ideal, $\delta:=v^\alpha v^{-1}$
satisfies $\varphi(\delta)=\sh(v^\alpha)\sh(v)^{-1}=\prod_i(1+\sigma_i)
^{e_i(r)}$. Now $H$ has exponent $p$, so $R\gamma_2(F)=F^p\gamma_2(F)$ and
every $e_i(r)$ is divisible by $p$, say $e_i(r)=pf_i$; as
$(1+\sigma)^p=1+\sigma^p$ in a commutative $\Fp$-algebra,
\[
\varphi(\delta)=\prod_i\bigl(1+\sigma_i^{\,p}\bigr)^{f_i},
\]
negative exponents being harmless since these are units. Hence
$\varphi(\delta)-1$ lies in the ideal generated by the
$\sigma_i^{\,p}\in\I(A)^{2p}$, so by \eqref{eq:PS} in $A$ we get
$\varphi(\delta)\in\bar V'\cap D_{2p}(\bar E')$. But
$\bar E'\cong(C_{p^2})^d$ by Proposition~\ref{prop:E}(1), and
$D_{2p}(\bar E')=\mho_2(\bar E')=1$ for an abelian group, since
$p<2p\le p^2$. Therefore $\varphi(\delta)=1$, that is
$\delta\in\ker(\varphi|_V)=C(V)$. Finally a source $v\in\mho_1(E)=W_p$ has
$v^\alpha v^{-1}\in W_{p+1}=1$ by Lemma~\ref{lem:reform}(1), and
$V=C(V)\oplus\mho_1(E)$.
\end{proof}

\begin{theorem}[dictionary]\label{thm:dict}
Let $p$ be odd and $H$ of exponent $p$ and class $c\le p-2$. Under the
identifications of Proposition~\ref{prop:E}(3):
\begin{enumerate}[label=\textup{(\arabic*)}]
\item $C(V)\cong\fv\cong H_2(\fl(H))$, the depth flag corresponding to the
degree filtration; in particular the Schur multiplier $M(H)$ satisfies
$\dim_{\Fp}M(H)/pM(H)=\dim H_2(\fl(H))$, and
$\dim H^2(H,\Fp)=\dim H_2(\fl(H))+d$;
\item for $\beta\in\Aut(H)$ with Lazard image $\theta\in\Aut(\fl)$, the
action $\rho(\beta)$ on $C(V)$ is the action of $\theta$ on $\fv$ of
Lemma~\ref{lem:Omega}\textup{(2)}; if $\theta=\exp(\mu)$ with
$\mu\in\Der(\fl)$ nilpotent, then $\rho(\beta)|_{C(V)}=\exp(\Om(\mu))$.
\end{enumerate}
\end{theorem}
\begin{proof}
(1) The first assertion is Proposition~\ref{prop:E}(3) and Hopf's formula.
Since $R/(R\cap F')$ is free abelian, $R/[F,R]\cong M(H)\oplus\Z^d$, so
$V=(R/[F,R])\otimes\Fp$ has dimension $\dim H^2(H,\Fp)$ and $C(V)$, the image
of $R\cap F'$, is $M(H)/pM(H)$. (2) Lift $\beta$ to $F$; the induced
automorphism of $E$ descends to $E^\circ$ and corresponds to an automorphism
of $\fe$ preserving $\fv$ and inducing $\theta$, so
Lemma~\ref{lem:Omega}(1),(4) apply.
\end{proof}

\section{The logarithm and the main theorem}\label{sec:log}

\subsection{The logarithm of a unipotent automorphism}

\begin{lemma}\label{lem:log}
Let $A$ be a finite-dimensional $\Fp$-algebra with an ideal $I$ such that
$A=\Fp1\oplus I$ and $I^p=0$, and let $\alpha\in\Aut(A)$ satisfy
$(\alpha-1)(I)\subseteq I^2$. Put $N=\alpha-1$ and
$\Delta=\sum_{k=1}^{p-2}\frac{(-1)^{k+1}}{k}N^k$. Then $N^{p-1}=0$, the map
$\Delta$ is a derivation of $A$ with $\Delta(I^n)\subseteq I^{n+1}$, and
$\alpha=\exp\Delta:=\sum_{k=0}^{p-2}\Delta^k/k!$.
\end{lemma}
\begin{proof}
First, $N(I^n)\subseteq I^{n+1}$: write an element of $I^n$ as a sum of
products of $n$ elements of $I$ and expand $\alpha(a_1\cdots a_n)-
a_1\cdots a_n$ telescopically, each summand containing a factor
$N(a_j)\in I^2$. Hence $N^{p-1}(I)\subseteq I^p=0$, and $N(1)=0$. In
$\Z_{(p)}[t]/(t^{p-1})$ the truncated identities $\exp(\log(1+t))=1+t$ and
$\log(\exp t)=t$ hold, all denominators being prime to $p$; substituting
$t\mapsto N$ gives $\alpha=\exp\Delta$ and $\Delta^{p-1}=0$, as well as
$\Delta(I^n)\subseteq I^{n+1}$.

For the Leibniz rule, let $Q$ be the quotient of $I\otimes I$ by the subspace
$\sum_{i+j\ge p}I^i\otimes I^j$. The multiplication $m:I\otimes I\to A$
factors through $Q$, and the commuting operators $X=\Delta\otimes1$,
$Y=1\otimes\Delta$ act on $Q$ raising the total degree by one, so that every
monomial $X^aY^b$ with $a+b\ge p-1$ vanishes on $Q$. Hence, on $Q$,
\[
\alpha\otimes\alpha=\exp(X)\exp(Y)=\sum_{a+b\le p-2}\frac{X^aY^b}{a!\,b!}
=\exp(X+Y),
\]
the last equality because $\exp(X)\exp(Y)=\exp(X+Y)$ holds in
$\Z_{(p)}[X,Y]$ modulo total degree $p-1$. Put
$M':=\alpha\otimes\alpha-1$ on $Q$; then $M'=\exp(X+Y)-1$ with
$(X+Y)^{p-1}=0$, so $\log(1+M')=X+Y$ on $Q$. On the other hand
$m\circ M'=N\circ m$, because
$m((\alpha\otimes\alpha)(a\otimes b))=\alpha(ab)$; therefore
$m\circ M'^k=N^k\circ m$ for all $k$ and $m\circ\log(1+M')=\Delta\circ m$.
Thus $\Delta(ab)=m((X+Y)(a\otimes b))=\Delta(a)b+a\Delta(b)$ for
$a,b\in I$, and trivially when $a$ or $b$ is a scalar.
\end{proof}

\subsection{The Lazard algebra inside the small group ring}
Let $p$ be odd, $G$ a group of exponent $p$ and class $<p$, and
$B:=KG/\I^{\,p}$. In $B$ the truncated series $\exp$ and $\log$ are mutually
inverse bijections between $\I(B)$ and $1+\I(B)$, and the
Baker--Campbell--Hausdorff identity $\exp(x)\exp(y)=\exp(\mathrm{BCH}(x,y))$
holds for $x,y\in\I(B)$: these are identities in
$\Z_{(p)}\langle\langle x,y\rangle\rangle$ modulo degree $p$, whose
denominators are prime to $p$. Since $G\cap(1+\I^p)=D_p(G)=1$, the group $G$
embeds in $B^\times$ and $\log:G\to\I(B)$ is injective; its image
$\fl_B:=\log G$ is closed under scalar multiples ($\log g^k=k\log g$) and
under BCH, hence, by the inverse BCH formulas of the Lazard correspondence
\cite[Ch.~10]{Khu98}, under addition and brackets. So $\fl_B$ is a Lie
subalgebra of $B$, and it is the Lazard algebra of $G$: normal subgroups $N$
correspond to the ideals $\log N$, and $\log g\in\I(B)^n$ if and only if
$g\in D_n(G)=\gamma_n(G)$, since $\log g=(g-1)u$ with $u$ a unit. In
particular $\fl_B\cap\I(B)^n=\gamma_n(\fl_B)$.

We apply this to $G=E^\circ$, writing $\fe=\log E^\circ\subseteq B$ and
$\fv=\log V^\circ$, a central ideal contained in $\I(B)^2$. Since
$x^p-1=(x-1)^p\in\I^p$ for $x\in E$, we have $B=KE/\I(KE)^p$, and
\[
s':=s(E,V)/\I(s)^p=B/J_B,\qquad J_B:=\text{the image of }\I(KV)\I(KE).
\]
For $v\in V^\circ$ one has $v-1=\log v\cdot w_v$ with $w_v\in1+\I(B)$, so
$J_B=\fv\cdot\I(B)$; and in $s'$ we get $(v-1)^2=0$ by \eqref{eq:trunc0},
whence $\log v=v-1$ there. For $1\le j\le p-1$ let
$\varepsilon_j:S^j\fe\to B$ be the symmetrization
$\varepsilon_j(u_1\cdots u_j)=\frac1{j!}\sum_{\pi\in S_j}u_{\pi(1)}\cdots
u_{\pi(j)}$; it is $\ad(\fe)$-equivariant, since $\ad_b$ is a derivation of
the associative product and Leibniz on a symmetrized product returns a
symmetrized sum.

\begin{lemma}[symmetrization decomposition]\label{lem:symm}
\begin{enumerate}[label=\textup{(\arabic*)}]
\item $B=\Fp\oplus\fe\oplus\bigoplus_{j=2}^{p-1}\varepsilon_j(S^j\fe)$, every
summand is $\ad(\fe)$-stable, and the decomposition is compatible with the
$\I$-adic filtration: each $\I(B)^n$ is the sum of its intersections with
the summands, and $\I(B)^n\cap\fe=\gamma_n(\fe)$.
\item $J_B=\bigoplus_{j=2}^{p-1}\varepsilon_j(\fv\cdot S^{j-1}\fe)$, a sum of
$\ad(\fe)$-stable subspaces of the summands with $j\ge2$.
\item Consequently $s'=\Fp\oplus\fe\oplus C$ with $C$ $\ad(\fe)$-stable; the
group $E^\circ$ and the algebra $\fe$ embed in $s'$,
$\fe\cap\I(s')^n=\gamma_n(\fe)$, and the projection $p_1:s'\to\fe$ along
$\Fp\oplus C$ satisfies $p_1([u,b])=[p_1(u),b]$ for $u\in s'$, $b\in\fe$, as
well as $p_1(\I(s')^n)\subseteq\gamma_n(\fe)$.
\end{enumerate}
\end{lemma}
\begin{proof}
(1) Filter $S^j\fe$ by total degree, induced from the lower central
filtration of $\fe$, and let
$\varepsilon:\bigoplus_{j\le p-1}S^j\fe\to B$ be $\sum_j\varepsilon_j$, a
filtered map. Elements of degree $\ge p$ map into $\I(B)^p=0$, so
$\varepsilon$ factors through the quotient by degree $\ge p$. On associated
graded objects $\gr\varepsilon$ is the symmetrization
$\bigoplus_{j<p}S^jL\to u(L)/(\deg\ge p)$ with $L=L(E^\circ)$. By the
restricted Poincar\'e--Birkhoff--Witt theorem \cite{Jac62} the length
filtration of $u(L)$ has associated graded $S(L)/(x^p)$, which agrees with
$S^jL$ in lengths $j<p$, and symmetrization is a section of
$U^{(j)}\to U^{(j)}/U^{(j-1)}$; hence $\bigoplus_{j<p}S^jL\to u(L)$ is
injective, and it is surjective in degrees $<p$ because a PBW monomial of
length $\ell$ in a graded basis of $L$ has degree $\ge\ell$. So
$\gr\varepsilon$ is an isomorphism and $\varepsilon$ is a strict isomorphism
of filtered spaces, which gives (1); the summand $j=1$ is $\fe$, and
$\I(B)^n\cap\fe=\gamma_n(\fe)$ was noted above.

(2) $J_B=\fv\cdot\I(B)=\sum_{j\ge1}\fv\cdot\varepsilon_j(S^j\fe)$. For
$z\in\fv$ central and $j+1\le p-1$ one has
$z\,\varepsilon_j(u_1\cdots u_j)=\varepsilon_{j+1}(zu_1\cdots u_j)$, since
$z$ may be inserted in any of the $j+1$ positions; and for $j=p-1$ the
product lies in $\I^2\cdot\I^{p-1}=0$. Stability holds because $\fv$ is an
ideal of $\fe$.

(3) As $s'=B/J_B$ and $J_B$ is a sum of subspaces of the summands with
$j\ge2$, the decomposition descends with $\fe$ intact; the filtration
statements descend because $\I(s')^n$ is the image of $\I(B)^n$ and the
decomposition is strict. Equivariance of $p_1$ is the $\ad(\fe)$-stability
of the summands.
\end{proof}

\subsection{The main theorem}

\begin{theorem}\label{thm:main}
Let $p$ be odd and $H$ of exponent $p$ and class $c\le p-2$. For every
$\alpha\in\UAut(KH)$ there is $\mu\in\Der_{>0}(\fl(H))$ such that
$\beta:=\exp(\mu)\in\Aut(H)$ lies in $\UAut(KH)$ and $\rho(\alpha)=
\rho(\beta)$ on $V$.
\end{theorem}
\begin{proof}
Let $\hat\alpha$ be a lift of $\alpha$ to $s=s(E,V)$; it preserves
$\{V-1\}$ and $\I(s)^p$, hence induces an automorphism of $s'$, still written
$\hat\alpha$, with $(\hat\alpha-1)\I(s')\subseteq\I(s')^2$. By
Lemma~\ref{lem:log}, $D:=\log\hat\alpha$ is a derivation of $s'$ with
$D(\I(s')^n)\subseteq\I(s')^{n+1}$ and $\hat\alpha=\exp D$; being a
polynomial in $\hat\alpha-1$, it preserves the image $\bar V$ of $\{V-1\}$
in $s'$. Now $\bar V=\{v-1:v\in V^\circ\}=\log V^\circ=\fv\subseteq\fe$ by
Lemma~\ref{lem:symm}(3), and $v\mapsto v-1$ is an isomorphism
$C(V)\cong V^\circ\to\fv$ of $\Fp$-spaces intertwining $\rho(\alpha)$ with
$\hat\alpha|_\fv$: indeed $\{v^\alpha-1\}=\hat\alpha\{v-1\}$, and
$v^\alpha\in C(V)$ by Proposition~\ref{prop:absorb}. Hence
\[
\rho(\alpha)|_{C(V)}=\exp\bigl(D|_\fv\bigr).
\]
Put $\tilde\mu:=p_1\circ D|_\fe:\fe\to\fe$. By equivariance of $p_1$,
\[
\tilde\mu[a,b]=p_1\bigl([Da,b]+[a,Db]\bigr)=[\tilde\mu a,b]+[a,\tilde\mu b],
\]
so $\tilde\mu$ is a derivation of $\fe$; it satisfies
$\tilde\mu(\gamma_n\fe)\subseteq p_1(\I(s')^{n+1})\subseteq
\gamma_{n+1}(\fe)$, and $\tilde\mu|_\fv=D|_\fv$ because
$D(\fv)\subseteq\fv\subseteq\fe$ and $p_1$ is the identity on $\fe$. Let
$\mu$ be the derivation of $\fl=\fe/\fv$ induced by $\tilde\mu$; then
$\mu(\fl)\subseteq\gamma_2(\fl)$, that is $\mu\in\Der_{>0}(\fl)$, and
$\tilde\mu^{c+1}=0$ since $\fe$ has class $\le c+1$.

The truncated exponentials $\exp(\mu)=\sum_{k\le c}\mu^k/k!$ and
$\exp(\tilde\mu)$ are therefore defined, and they are Lie algebra
automorphisms: expanding
$\exp(\tilde\mu)[a,b]=\sum_n\frac1{n!}\tilde\mu^n[a,b]$ by Leibniz and the
binomial identity gives $[\exp(\tilde\mu)a,\exp(\tilde\mu)b]$, all the
factorials involved being those of integers $\le c\le p-2$, hence
invertible. By the Lazard correspondence $\beta:=\exp(\mu)\in\Aut(H)$, and
$\beta$ is the identity on $H/\gamma_2(H)$, hence on $\I/\I^2$, so
$\beta\in\UAut(KH)$. By Theorem~\ref{thm:dict}(2),
\[
\rho(\beta)|_{C(V)}=\exp(\Om(\mu))=\exp\bigl(\tilde\mu|_\fv\bigr)
=\exp\bigl(D|_\fv\bigr)=\rho(\alpha)|_{C(V)} ,
\]
and both $\alpha$ and $\beta$ act trivially on $\mho_1(E)$
(Proposition~\ref{prop:absorb}); since $V=C(V)\oplus\mho_1(E)$, they agree
on $V$.
\end{proof}

\begin{theorem}[closed formula]\label{thm:formula}
Let $p$ be odd and $H$ of exponent $p$ and class $c\le p-2$. Then,
identifying $C(V)$ with $\fv=H_2(\fl(H))$,
\[
U(H)=\exp\Bigl(\Om\bigl(\Der_{>0}(\fl(H))\bigr)\Bigr)\ \le\
GL\bigl(H_2(\fl(H))\bigr),\qquad
|U(H)|=p^{\,\dim\Om(\Der_{>0}(\fl(H)))} .
\]
In particular $U(H)=1$ if and only if $\Om(\Der_{>0}(\fl(H)))=0$.
\end{theorem}
\begin{proof}
For $\mu\in\Der_{>0}(\fl)$ we saw that $\exp(\mu)\in\Aut(H)\cap\UAut(KH)$,
and $\rho(\exp\mu)|_{C(V)}=\exp(\Om(\mu))$ by Theorem~\ref{thm:dict}(2);
conversely every $\rho(\alpha)$ has this form by Theorem~\ref{thm:main}.
With Proposition~\ref{prop:absorb} this gives $U(H)=\exp(\fu)$ for
$\fu:=\Om(\Der_{>0}(\fl))$. Now $\fu$ is a Lie subalgebra of $\End(\fv)$
(Lemma~\ref{lem:Omega}(2)) all of whose elements raise the degree filtration
of $\fv$, which has length $\le c$; so $X^c=0$ for $X\in\fu$, and as
$c\le p-2$ the map $\exp$ is injective on $\fu$, with inverse $\log$. Hence
$|U(H)|=|\fu|$.
\end{proof}

\begin{corollary}[graded Lazard algebras]\label{cor:carnot}
Let $p$ be odd and $H$ of exponent $p$ and class $c\le p-2$, and suppose
$\fl(H)=\bigoplus_{i\ge1}\fl_i$ is graded, with $[\fl_i,\fl_j]\subseteq
\fl_{i+j}$ and generated by $\fl_1$. Then the Levi factorization holds, and
central Frattini extensions of $H$ whose small group rings are isomorphic as
pairs are isomorphic groups.
\end{corollary}
\begin{proof}
Then $\gamma_k(\fl)=\bigoplus_{i\ge k}\fl_i$ and $\gr\fl\cong\fl$ as graded
Lie algebras, so $L(H)=\gr\fl\cong\fl$. Given $\theta\in\Aut(KH)$, the map
$\gr\theta$ restricts to a graded Lie automorphism of $L(H)$
(Lemma~\ref{lem:grrigid}), hence to one of $\fl$, hence to some
$\beta\in\Aut(H)$ acting on $\fl_1=L_1=\I/\I^2$ as $\theta$ does; thus
$\beta^{-1}\theta\in\UAut(KH)$ and the Levi factorization holds. Now apply
Theorem~\ref{thm:main} and Lemma~\ref{lem:reform}(3).
\end{proof}

\begin{proof}[Proof of Theorem~\ref{thmA}]
(1) is Theorem~\ref{thm:main}; (2) is Theorem~\ref{thm:dict}(1),
Proposition~\ref{prop:absorb} and Theorem~\ref{thm:formula}; (3) is the last
clause of Theorem~\ref{thm:formula}; (4) is Corollary~\ref{cor:carnot}.
\end{proof}

\begin{remark}[the two hypotheses]\label{rem:hyp}
The bound $c\le p-2$ enters only through Proposition~\ref{prop:E}: the
universal extension must have class $<p$ for the Lazard correspondence to
apply to it. The hypothesis of exponent $p$ enters twice, in
Proposition~\ref{prop:absorb} and in the identification of $\fe$ with a
Lie algebra over $\Fp$. For exponent $p^2$ the logarithm of
Lemma~\ref{lem:log} still exists, but $KH/\I^{\,p}=K[H/D_p(H)]/\I^{\,p}$ sees
only the exponent-$p$ quotient, and what remains is the question of which
unipotent automorphisms of that quotient lift to $H$.
\end{remark}

\section{Examples}\label{sec:examples}

Throughout this section $p\ge5$, and $H$ has exponent $p$ and class
$c\le p-2$; we write $\fl=\fl(H)$ and $\fv=H_2(\fl)=\fr/[\ff,\fr]$. If $\fl$
is graded, every $\mu\in\Der_{>0}(\fl)$ is a sum of graded derivations
$\mu_j$ of degrees $j\ge1$, and $\Om(\mu)=\sum_j\Om(\mu_j)$ with
$\Om(\mu_j)$ raising the degree of $\fv$ by $j$.

\subsection{One layer}

\begin{corollary}\label{cor:onelayer}
If $\fv$ is concentrated in a single degree, then $U(H)=1$.
\end{corollary}
\begin{proof}
$\Om(\mu)$ raises the degree filtration of $\fv$, whose target is zero.
\end{proof}

\begin{example}[relatively free groups]\label{ex:ladder}
Let $d\ge2$, $2\le c\le p-2$ and $H=F_d/\gamma_{c+1}(F_d)F_d^{\,p}$. Then
$\fl$ is the free nilpotent Lie algebra of class $c$ on $d$ generators,
$\fr=\ff_{c+1}=\ell_{c+1}(d)$ and $[\ff,\fr]=0$, so $\fv=\ell_{c+1}$ is
concentrated in degree $c+1$. Hence $U(H)=1$ and
$\dim V=\ell_{c+1}(d)+d$; and $\fl$ is graded, so by
Corollary~\ref{cor:carnot} non-isomorphic central Frattini extensions of $H$
have non-isomorphic small group rings. As $c$ and $d$ vary, these extensions
realize every nilpotency class up to $p-1$, every rank, and exponent
dividing $p^2$. The case $d=c=2$ is the extraspecial group $p^{1+2}_+$ of
order $p^3$ and exponent $p$, whose central Frattini extensions have all
orders $p^4,\dots,p^7$; it is the exact odd counterpart of $H=D_8$ at $p=2$
\cite[Ex.~6.3]{HS07}, where instead $U(H)$ is as large as the flag of $V$
allows.
\end{example}

\subsection{Two generators}\label{subsec:twogen}

\begin{proposition}\label{prop:twogen}
Let $d=2$, $3\le c\le p-2$ and $N\le\gamma_c(F_2)$ with image
$\bar N\subseteq\ell_c(2)$, and put $H=F_2/\gamma_{c+1}(F_2)F_2^{\,p}N$.
Then $U(H)=1$, the Levi factorization holds, and non-isomorphic central
Frattini extensions of $H$ have non-isomorphic small group rings.
\end{proposition}
\begin{proof}
Here $\fl=\ff/\fr$ with $\fr=\bar N+\ff_{c+1}$, so $[\ff,\fr]=[\ell_1,
\bar N]$ and $\fv=\bar N\oplus\ell_{c+1}/[\ell_1,\bar N]$ lives in degrees
$c$ and $c+1$; and $\fl$ is graded. Only graded derivations of degree $1$ can
act nontrivially on $\fv$, and since $c\ge3$ gives $\fl_2=\ell_2$, such a
derivation is an arbitrary linear map $\ell_1\to\ell_2$, which extends to
$\ff$ and preserves $\fr$: thus $\Der_1(\fl)\cong\Hom(\ell_1,\ell_2)$, of
dimension $2$ because $\dim\ell_2(2)=1$. The inner derivations $\ad_x$ with
$x\in\ell_1$ already span a $2$-dimensional subspace of it, since
$[x,\ell_1]=0$ forces $x=0$, and they lie in $\ker\Om$
(Lemma~\ref{lem:Omega}(3)). Hence $\Om(\Der_{>0}(\fl))=0$, and
Theorem~\ref{thm:formula} with Corollary~\ref{cor:carnot} concludes.
\end{proof}

\begin{example}[maximal class of order $p^4$]\label{ex:p47}
Let $H=\langle a,b\mid a^p,\,b^p,\,[b,a,b],\,[b,a,a,a]\rangle$, the group of
maximal class, order $p^4$ and exponent $p$. Writing $w=[y,x]$ in
$\ell(x,y)$, this is the case $c=3$, $\bar N=\langle[w,y]\rangle$ of
Proposition~\ref{prop:twogen}: here $\fv_3=\langle[w,y]\rangle$ and
$\fv_4=\ell_4/\langle[[w,y],x],[[w,y],y]\rangle$ are both one-dimensional,
and for $\mu(x)=\lambda w$, $\mu(y)=\nu w$ one computes
\[
\Om(\mu)\bigl([w,y]\bigr)=\nu\,[[w,x],y]-\lambda\,[[w,y],y]
\pmod{[\ell_1,\bar N]} .
\]
The second term lies in $[\ell_1,\bar N]$, and so does the first, by Jacobi:
$[[w,x],y]=[[w,y],x]+[w,[x,y]]=[[w,y],x]$. Thus $U(H)=1$. Here the
commutator sector of $V$ occupies two layers of the depth flag, and a
single application of the Jacobi identity replaces the computation in
$u(L(E))$ that a direct approach through the flag requires.
\end{example}

\subsection{Class two}\label{subsec:class2}
Let $H$ have class $2$ and exponent $p$, with $L_1=H/\Phi(H)$,
$L_2=H'=\gamma_2(H)$ and commutator map $\omega:\Lambda^2L_1\to L_2$, whose
kernel we denote by $K$. Then $\fl=L_1\oplus L_2$ with bracket $\omega$, a
graded algebra; $\ff$ is free nilpotent of class $3$ on $L_1$, so
$\ff_2=\Lambda^2L_1$, $\fr=K\oplus\ff_3$, $[\ff,\fr]=[L_1,K]$, and
\[
\fv=K\ \oplus\ \ell_3(L_1)/[L_1,K]\qquad\text{in degrees $2$ and $3$.}
\]
Every linear map $\chi:L_1\to L_2$ is a derivation of $\fl$ (extend by zero
on $L_2$; both sides of Leibniz vanish, $L_2$ being central and containing
$[\fl,\fl]$), so $\Der_1(\fl)=\Hom(L_1,L_2)$; and $\exp(\chi)$ is the
central automorphism $g\mapsto g\chi(g)$ of $H$. Choosing a lift
$\tilde\chi:L_1\to\Lambda^2L_1$ of $\chi$,
\begin{equation}\label{eq:Omclass2}
\Om(\chi)\Bigl(\sum a_{ij}\,x_i\wedge x_j\Bigr)=
\sum a_{ij}\bigl([\tilde\chi x_i,x_j]+[x_i,\tilde\chi x_j]\bigr)
\pmod{[L_1,K]},\qquad \sum a_{ij}\,x_i\wedge x_j\in K .
\end{equation}

\begin{theorem}\label{thm:class2}
Let $p\ge5$ and let $H$ have class $2$ and exponent $p$. Then
$U(H)\cong\Hom(L_1,L_2)/\ker\Om$ as an elementary abelian group, every
element of $U(H)$ is realized by a central automorphism of $H$, the Levi
factorization holds, and non-isomorphic central Frattini extensions of $H$
have non-isomorphic small group rings. Moreover
$\ker\Om\supseteq\{\omega(x,\cdot):x\in L_1\}$, so that
\[
\log_p|U(H)|\ \le\ \dim L_1\cdot\dim L_2-\dim L_1
+\dim\operatorname{rad}\omega .
\]
\end{theorem}
\begin{proof}
By Theorem~\ref{thm:formula}, $U(H)=\exp\Om(\Hom(L_1,L_2))$; and
$\Om(\chi)^2=0$ because $\fv$ has two layers, so $\exp\Om(\chi)=1+\Om(\chi)$
and $U(H)$ is the additive group $\Om(\Hom(L_1,L_2))$, realized by the
central automorphisms $\exp(\chi)$. The Levi factorization and the last
assertion are Corollary~\ref{cor:carnot}. Finally the inner derivation
$\ad_x$ restricts to $\omega(x,\cdot)$ on $L_1$ and lies in $\ker\Om$, and
$x\mapsto\omega(x,\cdot)$ has kernel $\operatorname{rad}\omega$.
\end{proof}

\begin{example}[$p^{1+2}_+\times C_p$]\label{ex:extraCp}
Take $L_1=\langle a,b,t\rangle$, $L_2=\langle z\rangle$ with
$\omega(a,b)=z$ and $\omega(t,\cdot)=0$, so that
$K=\langle t\wedge a,\,t\wedge b\rangle$. Then $[L_1,K]$ is the part of
$\ell_3(a,b,t)$ of positive $t$-degree --- it contains $[[a,b],t]$ by Jacobi
--- so $\fv_3=\ell_3(a,b)$ has dimension $2$ and $\dim V=2+2+3=7$. For
$\chi$ with $\chi(t)=\gamma z$, formula \eqref{eq:Omclass2} gives
$\Om(\chi)(t\wedge a)=\gamma[[a,b],a]$ and
$\Om(\chi)(t\wedge b)=\gamma[[a,b],b]$, the contributions of $\chi(a)$ and
$\chi(b)$ having positive $t$-degree. Hence $U(H)\cong C_p$, generated by the
central automorphism $t\mapsto t[a,b]$, and the bound of
Theorem~\ref{thm:class2}, namely $3-3+1=1$, is attained.
\end{example}

\begin{proposition}[Heisenberg groups and direct powers]\label{prop:hd}
Let $p\ge5$.
\begin{enumerate}[label=\textup{(\arabic*)}]
\item Let $H=(p^{1+2}_+)^M$ and write $L_1=\bigoplus_{i=1}^MU_i$ with
$\dim U_i=2$, and $L_2=\bigoplus_iZ_i$ with $Z_i=\omega(\Lambda^2U_i)$. Then
$\fv_2=K=\bigoplus_{i<j}U_i\otimes U_j$, the space
$\fv_3=\bigoplus_i\ell_3(U_i)$ has dimension $2M$, the kernel
$\ker\Om=\bigoplus_i\Hom(U_i,Z_i)$ is exactly the space of inner
derivations, and $|U(H)|=p^{2M(M-1)}$.
\item For $H=\Heis(p^n)$, the Sylow $p$-subgroup of $SL_3(p^n)$,
$|U(H)|=p^{2n(n-1)}$.
\end{enumerate}
\end{proposition}
\begin{proof}
(1) Since $\Lambda^2L_1=\bigoplus_i\Lambda^2U_i\oplus\bigoplus_{i<j}
U_i\otimes U_j$ and $\omega$ is an isomorphism on the first part and zero on
the second, $K$ is as stated. Grade $\ell_3(L_1)$ by block multidegree in
$\mathbb N^M$. The pure components are the $\ell_3(U_i)$, of dimension
$\ell_3(2)=2$ each, and every mixed component lies in $[L_1,K]$: for
$u\in U_i$ and $w\in U_j$ with $i\ne j$ one has $[u,w]\in K$, which handles
the brackets $[[u,w],y]$, while for $u,u'\in U_i$ and $w\in U_j$ the Jacobi
identity gives $[[u,u'],w]=[[u,w],u']-[[u',w],u]$. As $[L_1,K]$ is
multihomogeneous and contained in the mixed part, $\fv_3=\bigoplus_i
\ell_3(U_i)$. Now lift $\chi\in\Hom(L_1,L_2)$ to $\tilde\chi$ with values in
$\bigoplus_i\Lambda^2U_i$, a complement of $K$, and write $\chi_{ij}$ for
the component $U_i\to Z_j$. For $v\in U_i$ and $w\in U_j$ with $i\ne j$ only
the pure terms survive in \eqref{eq:Omclass2}:
\[
\Om(\chi)(v\wedge w)=[\tilde\chi_{ij}(v),w]+[v,\tilde\chi_{ji}(w)]
\in\ell_3(U_j)\oplus\ell_3(U_i).
\]
The map $\Lambda^2U_j\otimes U_j\to\ell_3(U_j)$, $\zeta\otimes w\mapsto
[\zeta,w]$, is an isomorphism, both sides having dimension $2$; hence
$\Om(\chi)=0$ if and only if $\chi_{ij}=0$ for all $i\ne j$. So
$\ker\Om=\bigoplus_i\Hom(U_i,Z_i)$ has dimension $2M$ inside
$\dim\Hom(L_1,L_2)=2M^2$, which gives $|U(H)|=p^{2M(M-1)}$ by
Theorem~\ref{thm:class2}; and the inner derivations $\omega(x,\cdot)$ with
$x=\sum_ix_i$ are precisely the maps $\bigoplus_i\omega(x_i,\cdot)$, which
fill $\bigoplus_i\Hom(U_i,Z_i)$ because each $\omega|_{U_i}$ is
nondegenerate.

(2) Let $F=\mathbb F_{p^n}$. For $p$ odd the group $\Heis(F)$ is the
exponential of the strictly upper triangular $3\times3$ matrices over $F$,
so its Lazard algebra is the $3$-dimensional Heisenberg algebra
$\mathfrak h_F$ over $F$, regarded over $\Fp$. Every ingredient of
Theorem~\ref{thm:formula} is linear algebra compatible with extension of
scalars --- $\fv\otimes F=H_2(\fl\otimes F)$,
$\Der(\fl\otimes F)=\Der(\fl)\otimes F$ --- so
$\dim_F\Om_F(\Der_{>0}(\fl\otimes F))=\dim_{\Fp}\Om(\Der_{>0}(\fl))$. Since
$F\otimes_{\Fp}F\cong F^n$ through the Frobenius twists,
$\fl\otimes_{\Fp}F\cong\mathfrak h_F^{\oplus n}$ as $F$-Lie algebras, and
the computation in (1), valid over any field of odd characteristic, gives
$2n(n-1)$.
\end{proof}

\begin{remark}
So the unipotent freedom surviving in $\Heis(p^n)$ and in
$(p^{1+2}_+)^M$ is accounted for entirely by central automorphisms, and its
size counts the ``off-diagonal'' ones; for $M=1$ it recovers
$U(p^{1+2}_+)=1$. Note that Theorem~\ref{thm:class2} assumes nothing about
$\omega$ beyond surjectivity, so it covers every semifield group
\cite{Ver87}, and in particular the families of \cite{LW12}.
\end{remark}

\section{Open problems}\label{sec:open}

\begin{problem}[Levi factorization]\label{prob:levi}
Does $\Aut(KH)=\Aut(H)\cdot\UAut(KH)$ hold for every $H$ of exponent $p$ and
class $\le p-2$? Equivalently, is the image of $\Aut(KH)$ in
$GL(H/\Phi(H))$ contained in that of $\Aut(\fl(H))$? By
Corollary~\ref{cor:carnot} this holds whenever $\fl(H)$ is graded, and it is
the only hypothesis remaining in the conclusion about small group rings in
exponent $p$.
\end{problem}

\begin{problem}[the ceiling]\label{prob:ceiling}
Extend Theorem~\ref{thm:main} to class $c=p-1$. For
$H=F_d/\gamma_p(F_d)F_d^{\,p}$ the commutator sector of $V$ collapses into
depth $p$ and $U(H)=1$ follows from the flag alone, so the residual
configuration is that in which $H_2(\fl)$ is spread over at least three
degrees with top degree $p$, where it becomes invisible in
$s/\I^{\,p}$; the smallest instance is $p=5$, $c=4$, $d\ge3$, the case
$d=2$ being excluded as in Proposition~\ref{prop:twogen}.
\end{problem}

\begin{problem}[the kernel of $\Om$]
For $H$ of class $2$ and exponent $p$, describe $\ker\Om$ in terms of
$\omega$. Proposition~\ref{prop:hd} shows that it consists exactly of the
inner derivations when $\omega$ is an orthogonal sum of hyperbolic planes
with independent value lines; for which $\omega$ is this the case, and for
which is $U(H)=1$?
\end{problem}

\section*{Acknowledgements}

The author completed his B.Sc.\ in Mathematics at the Instituto de
Matem\'atica e Estat\'istica, Universidade de S\~ao Paulo, in 2026. No
funding was received for this work, and the author declares no conflicts of
interest.

During the preparation of this work the author used the large language model
Claude (Anthropic). The mathematical content arises from the author's own
research programme: the results, the examples, and the computer explorations
that guided their discovery are the author's. The model drafted the prose of
the present version of the manuscript, including the introduction, from the
author's earlier drafts and research notes, and proposed the organisation of
the paper; it supplied the first drafts of the proofs of Lemma~\ref{lem:log}
and Lemma~\ref{lem:symm}, rewrote several other proofs more concisely, and
identified and corrected errors in the author's earlier drafts; and it wrote
the code with which the author checked, independently of the proofs given
here, the dimension formulas and vanishing statements of
Section~\ref{sec:examples}. The author has reviewed and edited all of this,
has verified the mathematical arguments and the bibliographic data
independently, and takes full responsibility for the entire content of the
publication, including any part in which the tool was involved. The tool is
not an author.

\end{document}